\documentclass [12pt]{article}
\usepackage{graphicx,amssymb,amsfonts,latexsym,amsmath,amsthm,times}
\usepackage{epsfig}
\usepackage{fancyhdr}
\usepackage{color}
\usepackage[english]{babel}
\numberwithin{equation}{section}

\newtheorem{theorem}{Theorem}

\newtheorem{corollary}[theorem]{Corollary}

\newtheorem{definition}[theorem]{Definition}

\newtheorem{lemma}[theorem]{Lemma}

\newtheorem{proposition}[theorem]{Proposition}

\def\Rr{{\mathbb R}}

\def\H1{{\mathbb H}^1}

\def\Rn{{\mathbb R}^n}
\def\R2{{\mathbb R}^2}
\def\Rtd{{\mathbb R}^{2d}}
\def\Rfd{{\mathbb R}^{4d}}
\def\Rd{{\mathbb R}^d}

\def\S{{\cal S}}

\def\F{{\cal F}}

\def\L{\mathcal{L}}

\def\TM{{\widetilde{M}}}
\def\TV{{\widetilde{V}}}
\def\TF{{\widetilde{\mathcal F}}}
\def\supp{{\rm supp\,}}

\begin{document}


\vspace*{2cm} \normalsize \centerline{\Large \bf Boundedness  of Pseudo-Differential Operators on $L^p$,}
\normalsize \centerline{\Large \bf Sobolev,  and  Modulation Spaces}

\vspace*{1cm}

\centerline{\bf Shahla Molahajloo$^1$\footnote{Corresponding
author. E-mail: molahajloo@mast.queensu.ca}
 and  G\"otz E. Pfander$^2$ }

\vspace*{0.5cm}

\centerline{$^1$ Department of Mathematics and Statistics, Queen's 
University, Kingston, Ontario K7L3N6, Canada}

\centerline{$^2$ School of Engineering  and Science,
Jacobs  University,
28759 Bremen,
Germany}


\vspace*{1cm}

\noindent {\bf Abstract.}
We introduce new classes of modulation spaces over phase space. By means of the Kohn-Nirenberg correspondence, these spaces induce norms on pseudo-differential operators that bound their operator norms on $L^p$--spaces, Sobolev spaces, and modulation spaces.

\vspace*{0.5cm}




\vspace*{1cm}

\setcounter{equation}{0}
\section{Introduction}
Pseudo-differential operators are discussed in various areas of mathematics and mathematical physics,
for example, in partial differential equations, time-frequency analysis, and quantum mechanics  \cite{Hor1,Hor2, Kumano,WongPseudo, WongWeyl}. They are defined as follows.

Let $\sigma$ be a tempered distribution on phase space $\Rtd$, that is, $\sigma\in \S'(\Rtd)$ where  $\S(\Rtd)$ denotes the space of Schwartz class functions. The pseudo-differential operator $T_\sigma$
corresponding to the symbol $\sigma$  is given by
$$ T_{\sigma}f (x)=\int \sigma(x,\xi)\,\widehat{f}(\xi)\, e^{2\pi ix\cdot\xi}\,d\xi,\quad f\in \mathcal S(\Rd).  $$
 Here, $\widehat{ f}$ denotes the Fourier transform of $f$, namely,
$$\widehat f(\xi)=\F f (\xi)=\int f(x)\,  e^{-2\pi ix\cdot\xi}\,dx.$$

One of the central goals in the study of pseudo-differential operators is  to obtain  necessary and sufficient conditions for pseudo-differential operators to extend boundedly to function spaces such as $L^p(\Rd)$ \cite{CMW, CorderoNicolaPseudo, HwangLee, WongFred}. A classical result in this direction is the following.

For $m\in\Rr$, we let $S^m$ consist of all  functions  $\sigma$ in $C^\infty(\Rd{\times}\Rd)$
 such that for any multi-index $(\alpha,\beta)$, there is $C_{\alpha,\beta}>0$ with
 $$\big|\big(\partial_x^\beta\partial_\xi^\alpha\sigma\big)(x,\xi)\big|\leq C_{\alpha,\beta}(1+|\xi|)^{m-|\alpha|}.$$
 For $\sigma\in S^0(\Rd)$,  it is known that $T_\sigma$ acts  boundedly  on $L^p(\Rd)$, $p\in (1,\infty)$. A consequence of this result is
 that  if $\sigma\in S^m$, then $T_\sigma$ is a bounded operator mapping $H^p_{s+m}(\Rd)$ to $H^p_s(\Rd)$, where $H^p_s(\Rd)$ is the
 Sobolev Spaces  of order  $s\in\Rr$; for more details see Wong's book \cite{WongPseudo}.
 Similarly, in \cite{WongFred}, Wong obtains  weighted $L^p$--boundedness results for
pseudo-differential operators with symbols in $S^m$.

Smoothness and boundedness  of symbols though are far from being necessary (nor sufficient) for the $L^p$-boundedness of pseudo-differential operators.  In fact, every symbol $\sigma\in L^2(\Rtd)$ defines  a  so-called Hilbert--Schmidt
operator and Hilbert-Schmidt operators are bounded, in fact, compact operators on $L^2(\Rd)$.
Non-smooth and unbounded symbols have been considered systematically in the framework of modulation spaces, an approach that we continue in this paper.

Modulation spaces  were first introduced by Feichtinger in \cite{Fei2} and they have been further developed
by him and Gr\"ochenig in \cite{Fei1, Fei2, FeiGroch1, FeiGroch2, FeiGroch3, FeiGroch4}.
In the following, set $\phi(x)=e^{-\pi\|x\|^2/2}$ and let the dual
pair bracket $(\cdot,\cdot)$ be linear in the first argument and antilinear in the second argument.

\begin{definition}[Modulation spaces over Euclidean space]
Let $M_\nu$ denote  modulation by $\nu\in\Rd$, namely,  $M_\nu f(x)=e^{2\pi i t\cdot \nu}f(x)$, and let $T_t$ be translation by $t\in\Rd$, that is,  $T_t f (x)=g(x-t)$.

The short-time Fourier transform $V_\phi f$ of $f\in \mathcal S'(\Rd)$ with respect to the Gaussian window $\phi$ is given by
  \begin{eqnarray*}
    V_\phi f(t,\nu)=\F\big( f\,T_{t}\phi\big)(\nu)=(f,M_\nu T_t\phi)=\int f(x)\ e^{-2\pi i x \nu}\phi(x-t)\, dx\,.
  \end{eqnarray*}
The modulation space $M^{pq}(\Rd)$, $1\leq p,q \leq \infty$,  is a Banach space consisting of those $f\in \mathcal S'(\Rd)$ with
  \begin{eqnarray*}
   \|f\|_{M^{pq}} =\|V_\phi f \|_{L^{pq}}= \Big( \int\Big( \int |V_\phi f(t,\nu)|^p\, dt\Big)^{1/p}\, d\nu   \Big)^{1/q} < \infty\,,
  \end{eqnarray*}
with usual adjustment of the mixed norm space if $p=\infty$ and/or $q=\infty$.
\end{definition}

Roughly speaking, distributions in $M^{pq}(\Rd)$ `decay' at infinity like a function in $L^p(\Rd)$ and have the same local regularity as a function whose Fourier transform is in $L^q(\Rd)$.

The boundedness of pseudo-differential operators on modulation spaces are studied
for various classes of symbols, for example, in \cite{CorderoNicolaPseudo,Czaja, GrochHeil1, GrochHeil2, Toft1, Toft2, Toft3, Toft4}. In \cite{Toft1, Toft2} for example, Toft discusses boundedness of pseudo-differential operators on weighted modulation spaces. In \cite{CorderoNicolaPseudo}, Nicola and Cordero describe  a class of pseudo-differential operators with symbols $\sigma$ in modulation spaces for which $T_\sigma$ is bounded on $L^p(\Rd)$.

The modulation space membership  criteria on Kohn--Nirenberg symbols used in \cite{CorderoNicolaPseudo,Czaja, Toft1, Toft2} do not allow to require different decay in $x$ and $\xi$ of $\sigma(x,\xi)$. In the recently developed sampling theory for operators, though, a separate treatment of the decay of $x$ and $\xi$ was beneficial \cite{HongPfander, Pfander2, Pfander1}. In fact, this allows to realize canonical  symbol norms of convolution  and multiplication operators as modulation space norms on Kohn-Nirenberg symbols. Motivated by this work, we give the following definition.

\begin{definition}[Modulation spaces over phase space]

The symplectic Fourier transform of $F\in\S(\Rtd)$  is given by
\begin{eqnarray}
  \TF F(t,\nu)=\int_{\Rtd}e^{-2\pi i[(x,\xi),(t,\nu)]}F(x,\xi)\,dx\,d\xi, \label{eqn:symplectic}
\end{eqnarray}
where $[(x,\xi),(t,\nu)]$ is the symplectic form of $(x,\xi)$ and $(t,\nu)$ defined by
$[(x,\xi),(t,\nu)]=x\cdot\nu-\xi\cdot t.$
Analogously, symplectic modulation $\TM_{(t,\nu)}$ is  $\TM_{(t,\nu)} F(x,\xi)=e^{2\pi i[(x,\xi),(t,\nu)]}F(x,\xi)$.

The symplectic short-time Fourier transform $\TV_\phi f$ of $F\in \mathcal S'(\Rd)$ is given by
  \begin{eqnarray}
    \TV_\phi F(x,t,\xi,\nu)&=& \TF\big( F\,T_{(x,\xi)}\phi\big)(t,\nu)= (F,\TM_{(\nu,t)} T_{(x,\xi)}\phi) \label{eqn:Vs} \\
  \notag  &=&\iint e^{-2\pi i(\widetilde{x}\nu-\widetilde{\xi}t)}F(\widetilde{x},\widetilde{\xi})\varphi(\widetilde{x}-x,\widetilde{\xi}-\xi)\,
    d\widetilde{x}\,d\widetilde{\xi}.
  \end{eqnarray}
 The modulation space over phase space $\TM^{p_1p_2q_1q_2}(\Rtd)$, $1\leq p_1,p_2,q_1,q_2,\leq \infty$,  is the Banach space consisting of those $F\in \mathcal S'(\Rtd)$ with
  \begin{eqnarray}
  && \hspace{-1cm} \|F\|_{\TM^{p_1p_2q_1q_2}} =\|\TV_\phi F \|_{L^{p_1p_2q_1q_2}} \nonumber \\
  &=&\Big(\int \Big(\int \Big(\int \Big(\int |(\TV_{\psi} F)
(x,t,\xi,\nu)|^{p_1}\,dx\Big)^{p_2/p_1}\,dt\Big)^{q_1/p_2}\,d\xi\Big)^{q_2/q_1}\,d\nu\Big)^{1/q_1}\notag \\ &<&\infty\,,\label{Ms1}
  \end{eqnarray}
 with usual adjustments if $p_1=\infty$, $p_2=\infty$, $q_1=\infty$, and/or $q_2=\infty$.
\end{definition}

Note that the order of the list of variables in \eqref{eqn:Vs} is crucial as it indicates the order of integration in \eqref{Ms1}. We choose to list first the time variable $x$ followed by the time-shift variable $t$. The time variables are followed by the frequency variable $\xi$ and the frequency-shift variable $\nu$. Alternative orders of integration were  considered, for example, in  \cite{Shannon,CorderoNicolaPseudo,Toft1,Toft2}.

Below, $\L(X,Y)$ denotes the space of all bounded linear operators mapping the Banach space $X$ to the Banach space $Y$; $\L(X,Y)$ is equipped with the operator norm. Below, the conjugate exponent of $p\in [1,\infty]$ is denoted by $p'$. Our main result follows.
\begin{theorem}\label{MainTheorem}
Let $p_1,p_2,p_3,p_4,q_1,q_2,q_3,q_4\in[1, \infty]$. Then
 there exists $C>0$ such that
 \begin{equation}\label{MainTheorem1}
\|T_{\sigma}\|_{\L(M^{p_1q_1}, M^{p_2 q_2})}\leq C\, \|\sigma\|_{\TM^{p_3 p_4 q_3 q_4}},\quad \sigma\in \TM^{p_3 p_4 q_3 q_4}(\Rtd),
\end{equation}
if and only if
\begin{eqnarray}
\frac{1}{p'_1}+\frac{1}{p_2}\leq\frac{1}{p_3}+\frac{1}{p_4},&&p_4\leq\min\{p'_1,p_2\},\label{MainTheorem2}\\
\frac{1}{q'_1}+\frac{1}{q_2}\leq\frac{1}{q_3}+\frac{1}{q_4},&& q_4\leq\min\{q'_1,q_2\}.\label{MainTheorem3}
\end{eqnarray}
\end{theorem}
Theorem~\ref{MainTheoremWeight} below is a variant of Theorem~\ref{MainTheorem} that involves symbols in weighted modulation spaces.

Observe that  (\ref{MainTheorem2}) depends only on the  parameters $p_i$,  while (\ref{MainTheorem3})  depends analogously only on the  parameters $q_i$. That is, the conditions on  decay in time and on  decay in frequency, or, equivalently, on  smoothness  in frequency and on  smoothness in time, on the Kohn-Nirenberg symbol are linked to the respective conditions on domain and range of the operator, but time and frequency remain  independent
of one another. See Figure 1 for an illustration of conditions (\ref{MainTheorem2}) and (\ref{MainTheorem3}).

\begin{figure}
\thicklines
\begin{minipage}{7cm}
\setlength{\unitlength}{4.7cm}
\vspace{3cm}
\hspace{1cm}
\begin{picture}(1, 1.3)(-.2,-.5)
\put(0, 0){\vector(0, 1){1.3}}
\put(0, 0){\vector(1, 0){1.1}}

\put(1.12,-.015){\tiny{$\frac{1}{p_3}$}}
\put(-.015,1.34){\tiny{$\frac{1}{p_4}$}}
\put(1, 0){\line(0, 1){1}}
\put(0, 1){\line(1, 0){1}}
\put(0, 1){\circle*{.02}}
\put(-.05,.98){\tiny{1}}
\put(.25, .45){\line(-1, 1){.25}}
\put(0, .70){\circle*{.02}}
\put(-.22,.683){\tiny{$\frac{1}{p'_1}+\frac{1}{p_2}$}}
\put(.25, .45){\line(1, 0){.75}}
\put(1, 0){\circle*{.02}}
\put(.99,-.06){\tiny{1}}
\put(-.33,.44){\tiny{$\max\{\frac{1}{p'_1},\frac{1}{p_2}\}$}}
\put(0, .45){\circle*{.02}}
\multiput(0, .45)( 0.028,0){15}
{\circle*{.001}}
\multiput(.25,0)(0,0.03){16}
{\circle*{.001}}
\put(.13,-.07){\tiny{$\min\{\frac{1}{p'_1},\frac{1}{p_2}\}$}}
\put(.25,0){\circle*{.02}}
\put(.277, .45){\line(-1, 1){0.277}}
\put(.3, .45){\line(-1, 1){.30}}
\put(.32, .45){\line(-1, 1){.32}}
\put(.34, .45){\line(-1, 1){.34}}
\put(.36, .45){\line(-1, 1){.36}}
\put(.38, .45){\line(-1, 1){.38}}
\put(.40, .45){\line(-1, 1){.4}}
\put(.42, .45){\line(-1, 1){.42}}
\put(.44, .45){\line(-1, 1){.44}}
\put(.46, .45){\line(-1, 1){.46}}
\put(.48, .45){\line(-1, 1){.48}}
\put(.50, .45){\line(-1, 1){.50}}
\put(.52, .45){\line(-1, 1){.52}}
\put(.54, .45){\line(-1, 1){.54}}
\put(.56, .45){\line(-1, 1){.55}}
\put(.58, .45){\line(-1, 1){.55}}
\put(.60, .45){\line(-1, 1){.55}}
\put(.62, .45){\line(-1, 1){.55}}
\put(.64, .45){\line(-1, 1){.55}}
\put(.66, .45){\line(-1, 1){.55}}
\put(.68, .45){\line(-1, 1){.55}}
\put(.7, .45){\line(-1, 1){.55}}
\put(.72, .45){\line(-1, 1){.55}}
\put(.74, .45){\line(-1, 1){.55}}
\put(.76, .45){\line(-1, 1){.55}}
\put(.78, .45){\line(-1, 1){.55}}
\put(.8, .45){\line(-1, 1){.55}}
\put(.82, .45){\line(-1, 1){.55}}
\put(.84, .45){\line(-1, 1){.55}}
\put(.86, .45){\line(-1, 1){.55}}
\put(.88, .45){\line(-1, 1){.55}}
\put(.9, .45){\line(-1, 1){.55}}
\put(.92, .45){\line(-1, 1){.55}}
\put(.94, .45){\line(-1, 1){.55}}
\put(.96, .45){\line(-1, 1){.55}}
\put(.98, .45){\line(-1, 1){.55}}
\put(1, .45){\line(-1, 1){.548}}
\put(1, .47){\line(-1, 1){.53}}
\put(1, .49){\line(-1, 1){.51}}
\put(1, .51){\line(-1, 1){.49}}
\put(1, .53){\line(-1, 1){.47}}
\put(1, .55){\line(-1, 1){.45}}
\put(1, .57){\line(-1, 1){.43}}
\put(1, .59){\line(-1, 1){.41}}
\put(1, .61){\line(-1, 1){.39}}
\put(1, .63){\line(-1, 1){.37}}
\put(1, .65){\line(-1, 1){.35}}
\put(1, .67){\line(-1, 1){.33}}
\put(1, .69){\line(-1, 1){.31}}
\put(1, .71){\line(-1, 1){.29}}
\put(1, .73){\line(-1, 1){.27}}
\put(1, .75){\line(-1, 1){.25}}
\put(1, .77){\line(-1, 1){.23}}
\put(1, .79){\line(-1, 1){.21}}
\put(1, .81){\line(-1, 1){.19}}
\put(1, .83){\line(-1, 1){.17}}
\put(1, .85){\line(-1, 1){.15}}
\put(1, .87){\line(-1, 1){.13}}
\put(1, .89){\line(-1, 1){.11}}
\put(1, .91){\line(-1, 1){.09}}
\put(1, .93){\line(-1, 1){.07}}
\multiput(1, .935)( -.003,.003){23}
{\circle*{.0000000001}}
\multiput(1, .955)( -.003,.003){15}
{\circle*{.0000000001}}
\multiput(1, .975)( -.003,.003){9}
{\circle*{.0000000001}}
\end{picture}
\end{minipage}
\begin{minipage}{7cm}
\vspace{3cm}
\hspace{3cm}
\setlength{\unitlength}{4.7cm}
\thicklines
\begin{picture}(1, 1.3)(-.04,-.5)
\put(0, 0){\vector(0, 1){1.3}}
\put(1.12,-.015){\tiny{$\frac{1}{q_3}$}}
\put(0, 0){\vector(1, 0){1.1}}
\put(-.015,1.34){\tiny{$\frac{1}{q_4}$}}
\put(1, 0){\line(0, 1){1}}
\put(0, 1){\line(1, 0){1}}
\put(0, 1){\circle*{.02}}
\put(-.05,.98){\tiny{1}}
\put(.6, .6){\line(-1, 1){.4}}
\put(-.22,1.18){\tiny{$\frac{1}{q'_1}+\frac{1}{q_2}$}}
\put(0, 1.2){\circle*{.02}}
\multiput(.2,1)(-.02,0.02){11}
{\circle*{.001}}
\put(1, 0){\circle*{.02}}
\put(.99,-.06){\tiny{1}}
\thicklines
\put(.6,.6){\line(1,0){.4}}
\multiput(.6,0)(0,0.03){20}
{\circle*{.001}}
\multiput(.2,0)(0,0.03){34}
{\circle*{.001}}
\put(.6, 0){\circle*{.02}}
\put(.2, 0){\circle*{.02}}
\multiput(0,.6)(0.025,0){25}
{\circle*{.001}}
\put(0, .6){\circle*{.02}}
\put(-.33,.59){\tiny{$\max\{\frac{1}{q'_1},\frac{1}{q_2}\}$}}
\put(.5,-.07){\tiny{$\min\{\frac{1}{q'_1},\frac{1}{q_2}\}$}}
\put(.1,-.07){\tiny{$\frac{1}{q_2}-\frac{1}{q_1}$}}

\put(.58, .62){\line(1, 0){.42}}
\put(.56, .64){\line(1, 0){.44}}
\put(.54, .66){\line(1, 0){.46}}
\put(.52, .68){\line(1, 0){.48}}
\put(.50, .7){\line(1, 0){.5}}
\put(.48, .72){\line(1, 0){.52}}
\put(.46, .74){\line(1, 0){.54}}
\put(.44, .76){\line(1, 0){.56}}
\put(.42, .78){\line(1, 0){.58}}
\put(.40, .8){\line(1, 0){.6}}
\put(.38, .82){\line(1, 0){.62}}
\put(.36, .84){\line(1, 0){.64}}
\put(.34, .86){\line(1, 0){.66}}
\put(.32, .88){\line(1, 0){.68}}
\put(.30, .9){\line(1, 0){.7}}
\put(.28, .92){\line(1, 0){.72}}
\put(.26, .94){\line(1, 0){.74}}
\put(.24, .96){\line(1, 0){.76}}
\put(.22, .98){\line(1, 0){.78}}
\end{picture}
\end{minipage}
\vspace{-1.3cm}
\caption{
For fixed  $p_1,p_2$  and $q_1,q_2$, we mark the regions of $(\frac{1}{p_3},\frac{1}{p_4})$ and $(\frac{1}{q_3},\frac{1}{q_4})\ $ for which  every $\sigma\in \TM^{p_3 p_4 q_3 q_4}(\Rtd)$ induces  a bounded operator
$T_\sigma:M^{p_1 q_1}(\Rd)\to M^{p_2 q_2}(\Rd)$. In fact, for  $(\frac{1}{p_3},\frac{1}{p_4})$  and $(\frac{1}{q_3},\frac{1}{q_4})$ in the hashed  region, there exists   $C>0$  with $\|T_\sigma\|_{\L(M^{p_1 q_1},M^{p_2 q_2})}\leq C\, \|\sigma\|_{\TM^{p_3 p_4 q_3 q_4}}$.
  The conditions on the  time  decay parameters  $p_1,p_2,p_3,p_4$  are
 independent of the conditions on the frequency decay parameters
 $q_1,q_2,q_3,q_4$.
 }
\end{figure}
An $L^p$--boundedness  result for  the introduced classes of pseudo-differential operators  follows.
\begin{corollary}\label{Corollary}
Let
$p,p_3,p_4,q,q_3,q_4\in[1, \infty]$. Assume
$$\frac{1}{p'}+\frac{1}{q}\leq\frac{1}{p_3}+\frac{1}{p_4},\quad p_4\leq\min\{p',q\},$$
and
$$
\left\{
\begin{array}{lllll}
 \frac{1}{p}+\frac{1}{q}\!\!\!&\!\leq\frac{1}{q_3}+\frac{1}{q_4},& q_4\leq\min\{p,q\},&\textit{if }&   p,q\in[1,2],\\
 \frac{1}{p}+\frac{1}{q'}&\!\leq\frac{1}{q_3}+\frac{1}{q_4},& q_4\leq\min\{p,q'\},&\textit{if} & 1\leq p\leq 2\leq q,\\
\frac{1}{p'}+\frac{1}{q'}&\!\leq\frac{1}{q_3}+\frac{1}{q_4},& q_4\leq\min\{p',q'\},& \textit{if} & 2\leq\min\{ p,q\},\\
\frac{1}{p'}+\frac{1}{q}&\!\leq\frac{1}{q_3}+\frac{1}{q_4},& q_4\leq\min\{p',q\},& \textit{if} & 1\leq q\leq 2\leq p.
\end{array}
\right.
$$
Then $T_\sigma:L^p(\Rd)\to L^q(\Rd)$ is bounded and there exists a  constant $C>0$ such that
$$\|T_{\sigma}\|_{\L(L^p, L^q)}\leq C\,\|\sigma\|_{\TM^{p_3 p_4 q_3 q_4}},\quad \sigma\in \TM^{p_3 p_4 q_3 q_4}(\Rtd).$$
\end{corollary}

Corollary~\ref{Corollary} encompasses, for example, the  space of Hilbert--Schmidt operators on $L^2(\Rd)$,  namely
$$HS\big(L^2(\Rd)\big)=\big\{T_\sigma:\sigma\in \TM^{2,2,2,2}(\Rtd)=L^2(\Rtd)\big\}\subset \L(L^2(\Rd),L^2(\Rd)).$$
Moreover, Corollary~\ref{Corollary}  reconfirms also $L^2$--boundedness of  Sj\"ostrand class  operators \cite{Sjostrand1,Sjostrand2},
$$Sj\subset\{T_\sigma:\sigma\in \TM^{\infty,1,\infty,1}(\Rtd)\}\subset \L\Big(L^2(\Rd),L^2(\Rd)\Big).$$
Using the weighted version of Theorem~\ref{MainTheorem}, namely, Theorem~\ref{MainTheoremWeight}, we get the following boundedness result for Sobolev spaces.
\begin{corollary}
Let $p_1,p_2,p_3,p_4\in[1, \infty]$ and $s\in \Rr$. Let  $w$ be a moderate weight function on ${\Rr}^{4d}$ satisfying
$$w(x,t,\nu,\xi)\leq \big(1+|\xi|^2\big)^{s/2} \big(1+|\nu+\xi|^2\big)^{s/2},\quad x,t,\nu,\xi\in\Rd.$$
Assume that
$$\frac{1}{p'_1}+\frac{1}{p_2}\leq\frac{1}{p_3}+\frac{1}{p_4},\quad p_4\leq\min\{p'_1,p_2\}.$$
Then
$$
\|T_{\sigma}\|_{\L(H_s^{p_1}, H_s^{p_2})}\leq C\, \|\sigma\|_{\TM^{p_3,p_4,1,1}_w},\quad \sigma\in \TM_w^{p_3,p_4,1,1}(\Rtd),
$$
for some constant $C>0$.
\end{corollary}
The paper is structured as follows.   Section 2 discusses mixed norm spaces  and modulation spaces over Euclidean and over phase space in some detail. In Section 3, our boundedness  results for pseudo-differential operators with symbols in modulation spaces over phase space   are compared to results in the literature.  Finally, in Section 4  we prove our main results, Theorem~\ref{MainTheorem}, Corollary~\ref{Corollary}, and Theorem~\ref{MainTheoremWeight}.
\section{ Background on modulation spaces}\label{PhaseModulation}
\setcounter{equation}{0}

In the following, $x,\xi,t,\nu$ denote $d$-dimensional Euclidean variables. If not indicated differently, integration is with respect to the Lebesgue measure on $\Rd$.

Let $r=(r_1, r_2, \dots, r_n)$  where $1\leq r_i<\infty$, $i=1,2,\dots, n$. The mixed norm space $L^r({\Rr}^{n})$  is the set of all measurable functions $f$ on ${\Rr}^{n}$ for which
{\small{
\begin{eqnarray}
&&\hspace{-1cm}\|F\|_{L^r}=\nonumber\\
&&\Big(\int_{\Rr}\dots\Big(\int_{\Rr}\Big(\int_{\Rr}|F(x_1,
\dots,x_n)|^{r_1}\,dx_1\Big)^{r_1/r_2}\,dx_2\dots\Big)^{r_n/r_{n-1}}\,dx_n\Big)^{1/r_n}\nonumber
\end{eqnarray}}}
is finite \cite{benedek}.
$L^r({\Rr}^{n})$ is a Banach space with norm $\|\cdot\|_{L^r}$.
 Similarly, we define $L^r(\Rn)$ where $r_i=\infty$ for some indices $i$.

If $n=2d$, $r_1=r_2=\cdots=r_d=p$ and $r_{d+1}=\cdots=r_{2d}=q$, then we denote   $L^r(\Rtd)$   by $L^{pq}(\Rtd)$. Similarly, if $n=4d$ and $r_1=r_2=\cdots = r_d=p_1$, $r_{d+1}=\cdots=r_{2d}=p_2$, $r_{2d+1}=\cdots=r_{3d}=p_3$ and $r_{3d+1}=\cdots=r_{4d}=p_4$, we write $L^{p_1 p_2 p_3 p_4}(\Rfd)=L^r(\Rfd)$.

Let $w$ be a nonnegative measurable function on $\Rn$. We define  $L^r_w(\Rn)$ to be the space all   $f$ on $\Rn$ for
which $wf$ is in $L^r(\Rn)$.  $L^r_w(\Rn)$ is a Banach space with norm given by
$$
\|f\|_{L^r_w}=\|wf\|_{L^r}.
$$

In time-frequency analysis, it is advantageous to consider moderate weight functions $w$.  To define these, let ${\Rr}^+_0$ be the set of all nonnegative points in $\Rr$. Any locally integrable function   $v:\Rn\to{\Rr}_0^+$  with
$$
    v(x+y)\leq v(x)v(y)
$$
is called submultiplicative. Moreover, if $w:\Rn\to {\Rr}^+_0$ is locally integrable with
$$
    w(x+y)\leq C w(x) v(y),
$$
$C>0$, and $v$  submultiplicative, then $w$ is called moderate.

The short-time Fourier transform  of a tempered distribution $f\in\S'(\Rn)$
with respect to the window $\psi\in\S(\Rn)$ is given by
$$ V_\psi f (x,\xi)=\F(fT_x\overline\psi)(\xi)=(f,M_\xi T_x\psi)$$
where $M_\xi$ and $T_x$ denote modulation and translation as defined above.

With $\phi(x)=e^{-\pi\|x\|^2/2}$,  $w$   moderate   on ${\Rr}^{2d}$, and $p,q\in[1,\infty]$, the modulation space
$M_w^{{pq}}(\Rd)$ is the set of all tempered distributions $f\in\S'(\Rd)$ such that
$$
    V_\phi f\in L_w^{pq}(\Rtd).
$$
 with respective Banach space norm. Clearly, if $w\equiv 1$, then  $M_w^{pq}(\Rd)=M^{pq}(\Rd)$. Moreover, for any $s\in\Rr$ let
$$
    w_s(x,\xi)=\Big(1+|\xi|^2\Big)^{s/2}
$$
and denote $M^{pq}_{w_s}(\Rd)$ by $M^{pq}_s(\Rd)$.

Note that replacing the Gaussian function $\phi$ in the definition of modulation spaces by any other $\psi\in\S(\Rd)\setminus\{0\}$ defines the same space and an equivalent norm, a fact that will be used extensively below.

 Recall that the Sobolev space $H_s^p(\Rd)$ consist of all tempered distributions $u\in\S'(\Rd)$ for which $\|u\|_{H_s^p}=\|T_{w_s}u\|_{L^p}<\infty$ \cite{Toft1}. For any $s\in\Rr$ and $1\leq q\leq p\leq r\leq q'\leq \infty$ we have
\begin{equation}\label{IncluSobol1}
M_s^{pq}(\Rd)\subseteq H_s^r(\Rd),
\end{equation}
and for some  $C>0$,
$$\|f\|_{H_s^r}\leq C\|f\|_{M_s^{pq}}, \quad f\in M_s^{pq}(\Rd).$$
Similarly, $1\leq q'\leq r\leq p\leq q\leq \infty$ implies
\begin{equation}\label{IncluSobol2}
H_s^r(\Rd)\subseteq M_s^{pq}(\Rd),
\end{equation}
and for some constant $C>0$,
$$\|f\|_{M_s^{pq}}\leq C\|f\|_{H_s^r}, \quad f\in H_s^{r}(\Rd).$$
Let $\F L^p(\Rd)$ be the space of all tempered distributions $f$ in $\S'(\Rd)$ for which
there exists a function $h\in L^p(\Rd)$ such that $\hat h=f$. Then $\F L^p(\Rd)$ is a Banach space equipped with the norm
$$\|f\|_{\F L^p}=\|h\|_{L^p}.$$
The following lemma shows that  modulation space norms of compactly supported or bandlimited functions  can be estimated using $\mathcal F L^p$ and $L^p$ norms respectively  \cite{Oko,ConToft,Fei1,Toft}.

\begin{proposition}\label{compactMpq}
For $K\subset \Rd$  compact and $p,q\in[1,\infty]$, there are constants $A,B,C,D>0$ with
\begin{itemize}
\item[(i)] \quad $\displaystyle A\|f\|_{\F L^q}\leq \|f\|_{M^{pq}}\leq B\|f\|_{\F L^q},\quad f\in\S'(\Rd)$ with  $supp\ f\subseteq K$;

\item[(ii)] \quad $\displaystyle C \|f\|_{L^p}\leq\|f\|_{M^{pq}}\leq D\|f\|_{L^p},\quad f\in\S'(\Rd)$ with $supp\ \widehat{f}\subseteq K.$
\end{itemize}
\end{proposition}
In the following, we shall denote  norm equivalences as in statement $(i)$  above  by
$$\|f\|_{\F L^q}\asymp\|f\|_{M^{pq}},\quad  f\in\S'(\Rd),\quad\supp f\subseteq K.$$
Similarly, statement $(ii)$ becomes
$$\|f\|_{L^p}\asymp\|f\|_{M^{pq}}, \quad f\in\S'(\Rd), \quad\supp \widehat{f}\subseteq K.$$

The symplectic Fourier transform of $F\in\S(\Rtd)$ given in \eqref{eqn:symplectic} is a $2d$-dimensional Fourier transform followed by a rotation of phase space by $\frac{\pi}{2}$. This implies that the symplectic Fourier transform shares most properties with the Fourier transform, for example, Proposition~\ref{compactMpq} remains true when replacing the Fourier transform by the symplectic Fourier transform.

Let $p_1,p_2,q_1,q_2\in [1,\infty]$ and let $w$ be a  $v$--moderate weight function on ${\Rr}^{4d}$.
The weighted     modulation
space over phase space $\TM_w^{p_1 p_2 q_1 q_2}(\Rtd)$ is the set of
all tempered distributions $F\in\S'(\Rtd)$ for which $\TV_\psi F\in L_w^{p_1 p_2 q_1 q_2}(\Rfd)$.

Recapitulate that for $F\in \mathcal S'(\Rtd)$, we have $\TV_{\psi} F
(x,t,\xi,\nu)= V_{\psi} F
(x,\xi,\nu,-t)$,
  \begin{eqnarray}
  &&\hspace{-1cm} \|F\|_{\TM^{p_1p_2q_1q_2}} =\|\TV_\phi F \|_{L^{p_1p_2q_1q_2}} \nonumber \\
  &=&\Big(\int \Big(\int \Big(\int \Big(\int |\TV_{\psi} F
(x,t,\xi,\nu)|^{p_1}\,dx\Big)^{p_2/p_1}\,dt\Big)^{q_1/p_2}\,d\xi\Big)^{q_2/q_1}\,d\nu\Big)^{1/q_1},\notag \\ \notag \label{Ms2}
  \end{eqnarray}
 and
   \begin{eqnarray}
  &&\hspace{-1cm} \|F\|_{M^{p_1q_1q_2p_2}} =\|V_\phi F \|_{L^{p_1q_1q_2p_2}} \nonumber \\
  &=&\Big(\int \Big(\int \Big(\int \Big(\int |V_{\psi} F
(x,\xi,\nu,t)|^{p_1}\,dx\Big)^{q_1/p_1}\,d\xi\Big)^{q_2/q_1}\,d\nu\Big)^{p_1/q_1}\,dt \Big)^{1/p_1},\notag \\ \notag \label{Ms3}
  \end{eqnarray}
 with usual adjustments if $p_1=\infty$, $p_2=\infty$, $q_1=\infty$, and/or $q_2=\infty$.
This shows that the definition of
$\TM^{p_1,p_2,q_1,q_4}(\Rtd)$ is based on changing  the order of integration and on relabeling the integration exponents accordingly. Mixed $L^p$
spaces are sensitive towards the order of integration, and, hence  $\TM^{p_1 p_2 q_1 q_2}(\Rtd)\nsubseteq M^{p_1 p_2 q_1 q_2}(\Rtd)$
and $M^{p_1 p_2 q_1 q_2}(\Rtd)\nsubseteq \TM^{p_1 p_2 q_1 q_2}(\Rtd)$
 in general.
But for $1\leq p\leq q\leq \infty$, Minkowski's inequality
$$\Big(\int \Big(\int |F(x,y)|^p\,dx\Big)^{q/p}\,dy\Big)^{p}
\leq \Big(\int \Big(\int |F(x,y)|^q\,dy\Big)^{p/q}\,dx\Big)^{q}$$
(with adjustments for $p=\infty$ and/or $q=\infty$ holds and implies the following.
\begin{proposition}\label{AddConditions}
Let
$p_1,p_2,q_1,q_2\in[1, \infty]$ and $w$ be a moderate weight function on ${\Rr}^{4d}$.
\begin{itemize}
\item[(a)] If
$p_2\leq\min\{q_1,q_2\},$
then $ M_w^{p_1 q_1 q_2 p_2}(\Rtd)\subseteq \TM_{w}^{p_1 p_2 q_1 q_2}(\Rtd)$ and\\ $\|\sigma\|_{\TM_{w}^{p_1 p_2 q_1 q_2}}\leq \|\sigma\|_{M_w^{p_1 q_1 q_2 p_2}}.$
\item[(b)] If
$\max\{q_1,q_2\}\leq p_2,$
then $\TM_{w}^{p_1 p_2 q_1 q_2}(\Rtd)\subseteq M_w^{p_1 q_1 q_2 p_2}(\Rtd)$ and \\ $\|\sigma\|_{M_w^{p_1 q_1 q_2 p_2}}\leq \|\sigma\|_{\TM_{w}^{p_1 p_2 q_1 q_2}}.$
\end{itemize}
\end{proposition}
Note that results similar to ours could
also be achieved using symbols in $M_w^{p_3 p_4 q_3 q_4}(\Rtd)$, but the so obtained results would be weaker and they would necessitate the additional  condition  $p_4\leq\min\{q_3,q_4\}$ .

The modulation space over phase space$\TM_{w}^{p_1 p_2 q_1 q_2}(\Rtd)$ shares most of the properties of  ordinary modulation spaces. For example, if
$p_1\leq \widetilde{p}_1$, $p_2\leq \widetilde{p}_2$, $q_1\leq \widetilde{q}_1$ and $q_2\leq \widetilde{q}_2$, then
\begin{eqnarray}
  \TM_{w}^{p_1 p_2 q_1 q_2}(\Rtd)\subseteq \TM_{w}^{\widetilde{p}_1 \widetilde{p}_2 \widetilde{q}_1 \widetilde{q}_2}(\Rtd), \label{eqn:embedding}
\end{eqnarray}
and
$$\|\sigma\|_{\TM_{w}^{\widetilde{p}_1,\widetilde{p}_2,\widetilde{q}_1,\widetilde{q}_2}}\leq \|\sigma\|_{\TM_{w}^{p_1 p_2 q_1 q_2}},\quad \sigma\in \TM_{w}^{p_1,p_2,q_2,q_2}(\Rtd).$$
 Furthermore, let
 $p_1,p_2,q_1,q_2\in [1,\infty]$. Then the dual of $\TM_{w}^{p_1 p_2 q_1 q_2}(\Rtd)$
is $\TM_{w}^{p'_1 p'_2 q'_1 q'_2}(\Rtd)$ where $p'_1$, $p'_2$, $q'_1$, $q'_2$ are conjugate exponenets of
$p_1,p_2,q_1,q_2$ respectively.

The proofs of these results for modulation spaces over phase space  are similar to the ones for the ordinary modulation spaces \cite{Groch2},  and   are omitted.
\section{ Comparison of Theorem~\ref{MainTheorem} to results in the literature}
\setcounter{equation}{0}

Cordero and Nicola  as well as  Toft  proved the following theorem on  $M^{pq}$--boundedness for the class of pseudo-differential operators  with
symbols in $M^{s_1s_1s_2s_2}(\Rtd)$, see Theorem 5.2 in \cite{CorderoNicolaPseudo} and Theorem 4.3 in \cite{Toft1}.
\begin{theorem}\label{Nicola}
Let $p,q,s_1,s_2\in[1,\infty]$. Then
 for some  $C>0$,
\begin{eqnarray}\label{eqn:them}
  \|T_{\sigma}\|_{\L(M^{pq},M^{pq})}\leq C\, \|\sigma\|_{M^{s_1,s_1,s_2,s_2}},\quad \sigma\in M^{s_1,s_1,s_2,s_2}(\Rtd),
\end{eqnarray}
if and only if
 $$
 s_2\leq\min\{p,p',q,q',s_1'\}.
$$
\end{theorem}

Roughly speaking, to apply Theorem~\ref{Nicola}, we need to ensure  that $\sigma(x,\xi)$ has $L^s$ `decay' in $x$ and $\xi$ and that $\mathcal F\sigma(\nu,-t)=\mathcal F_s\sigma(t,\nu)$ has $L^{\min\{p,p',q,q',s'\}}$ `decay' in $t$ and $\nu$.  To apply Theorem~\ref{MainTheorem}, it suffices to ensure  that $\sigma(x,\xi)$ has $L^{s_1}$ `decay' in $x$ and $L^{s_2}$ `decay'$\xi$, and that $\mathcal F_s\sigma(t,\nu)$ has $L^{\min\{p,p',s_1'\}}$ `decay' in $t$ and $L^{\min\{q,q',s_2\}}$ `decay' in $\nu$.

Using embeddings such as \eqref{eqn:embedding}, we observe that indeed Theorem~\ref{Nicola} provides boundedness of $T_\sigma$ if and only if
\begin{eqnarray}
  \sigma\in \bigcup_{s=\max\{p,p',q,q'\}}^{\infty} M^{s,s,s',s'} \subseteq \bigcup_{s=\max\{p,p',q,q'\}}^{\infty} \TM^{s,s',s,s'} \label{eqn:inclusion}
\end{eqnarray}
while Theorem~\ref{MainTheorem} provides boundedness of $T_\sigma$ if and only if
$$
    \sigma\in \bigcup_{s_1=\max\{p,p'\}}^{\infty}\ \bigcup_{s_2= \max\{q,q'\} }^{\infty} \TM^{s_1,s_1',s_2,s_2'}\,.
$$
To obtain the set inclusion in \eqref{eqn:inclusion}, we used Theorem~\ref{AddConditions} and the fact that $s\geq \max\{p,p'\}$ implies $s\geq 2\geq s'$.

As $L^2=M^{2,2}$,  Theorem~\ref{Nicola} implies the following $L^2$--boundedness result.
\begin{corollary}\label{L2bounded}
Let $r,s\in[1,\infty]$. Then
 for some  $C>0$,
$$\|T_{\sigma}\|_{\L(L^2,L^2)}\leq C\, \|\sigma\|_{M^{r,r,s,s}},\quad \sigma\in M^{r,r,s,s}(\Rtd),$$
if and only if
 $$
 s\leq\min\{2,r'\}.
$$
\end{corollary}
Corollary~\ref{L2bounded} has been obtained earlier in 2003 by Gr\"ochenig and Heil \cite{GrochHeil1}.
As comparison, we formulate the respective consequence of Theorem~\ref{MainTheorem}.
\begin{corollary}\label{Corollary2}
For $r,s\in[2,\infty]$, there exists a  constant $C>0$ such that
$$\|T_\sigma\|_{\L(L^2,L^2)}\leq C\, \|\sigma\|_{\TM^{r,r',s,s'}},\quad \sigma\in \TM^{r,r',s,s'}(\Rtd).$$
\end{corollary}
As example, note that Theorem~\ref{Nicola} does not imply that $T_\sigma:L^2(\Rd)\to L^2(\Rd)$ is bounded for $\sigma\in M^{\infty,2,2,1}(\Rtd)$. But as $M^{\infty,2,2,1}(\Rtd)\subseteq \TM^{\infty,1,2,2}(\Rtd)$, Theorem~\ref{MainTheorem} indeed implies boundedness of $T_\sigma$ in this case.

For compositions of product   and convolution operators, Theorem~\ref{MainTheorem} implies the following result.

\begin{corollary}\label{Convolution}
For $p,q\in[2,\infty]$, let   $h_1\in M^{p,q'}(\Rd)$ and $h_2\in M^{p',q}(\Rd)$. Define
$$Tf=h_1\cdot(h_2*f),\quad f\in L^2(\Rd),$$
and
$$Hf=(h_1\cdot f)*h_2,\quad f\in L^2(\Rr).$$
Then $T$ and $H$ are  bounded operators on $L^2$ and moreover, there exist positive constants $C$ and $C'$ such that
$$\|T\|_{\L(L^2,L^2)}\leq C\, \|h_1\|_{M^{p,q'}}\|h_2\|_{M^{p',q}},$$
and
$$\|H\|_{\L(L^2,L^2)}\leq C'\, \|h_1\|_{M^{p,q'}}\|h_2\|_{M^{p',q}}.$$
\end{corollary}

The proof of Corollary~\ref{Convolution} follows immediately from Corollary~\ref{Corollary2}, Lemma~\ref{eta} and Lemma~\ref{h1h2}.
Note that not separately, the convolution and multiplication operators above may not be bounded operators.

\section{Proof  of  Theorem~\ref{MainTheorem}, Corollary~\ref{Corollary}, and Theorem~\ref{MainTheoremWeight}}
\setcounter{equation}{0}
\subsection{Proof  of  Theorem~\ref{MainTheoremWeight} and thereby of (\ref{MainTheorem2}) and (\ref{MainTheorem3}) implies (\ref{MainTheorem1}) in Theorem~\ref{MainTheorem} }
In this section we prove the weighted version of one implication of Theorem~\ref{MainTheorem}, that is the following theorem.
\begin{theorem}\label{MainTheoremWeight}
Let $w_1,w_2$ be moderate weight functions on $\Rtd$ and $w$ be a  moderate weight function on  ${\Rr}^{4d}$ that satisfies
\begin{equation}\label{WeightCondition}
w(x,t,\nu,\xi)\leq w_1(x-t,\xi)w_2(x,\nu+\xi).
\end{equation}
Let
$p_1,p_2,p_3,p_4,q_1,q_2,q_3,q_4\in[1, \infty]$ be such that
$$
\begin{array}{cc}
\frac{1}{p'_1}+\frac{1}{p_2}\leq\frac{1}{p_3}+\frac{1}{p_4},& p_4\leq\min\{p'_1,p_2\},\\
\frac{1}{q'_1}+\frac{1}{q_2}\leq\frac{1}{q_3}+\frac{1}{q_4},& q_4\leq\min\{q'_1,q_2\}.
\end{array}$$
Then
 there exists  a  constant $C>0$ such that
$$
\|T_{\sigma}\|_{\L(M_{w_1}^{p_1 q_1}, M_{w_2}^{p_2 q_2})}\leq C\, \|\sigma\|_{\TM_{w}^{p_3 p_4 q_3 q_4}},\quad \sigma\in \TM_{w}^{p_3 p_4 q_3 q_4}(\Rtd).
$$
\end{theorem}
To  prove Theorem~\ref{MainTheoremWeight} we need some preparation.
For functions $f$ and $g$ in $\S(\Rd)$, the Rihaczek transform $R(f,g)$
of $f$ and $g$ is defined by
$$R(f,g)(x,\xi)=e^{2\pi i x\cdot\xi}\hat f(\xi)\overline{g(x)}.$$
For  $\sigma\in\S(\Rtd)$, pseudo-differential operators are related to Rihaczek transforms by
$$\big(T_\sigma f,g\big)=\big(\sigma,\overline{R(f,g)}\big)$$
for all functions $f$ and $g$ in $\S(\Rd)$.
We define $A$, $T_A$  by $$(T_AF)(x,t)=F(A(x,t))=F(x-t,x)\,.$$
Then
$$\overline{R(f,g)}(x,\xi)=\F_{t\to\xi}\big(T_{A}(\overline{ f}{\otimes} g)(x,\cdot)\big),$$
where
$$\F_{t\to\xi}f(\cdot+x) =\int e^{-2\pi it\xi}f(t+x)\,dt.$$


\begin{lemma}\label{lemmaT_A}
Let $\varphi$ be a real valued  Schwartz function on $\Rd$. Then for all $f$ and $g$ in $\S(\Rd)$
$$ V_{T_A(\varphi{\otimes}\varphi)}T_A(\overline{ f}{\otimes} g)\, (x,t,\nu,\xi)=
\overline{ V_\varphi f (x-t,\xi)}\, V_{\varphi}g(x,\nu+\xi).
 $$
\end{lemma}

\begin{proof}
We compute
\begin{eqnarray}
&&\hspace{-1cm} V_{T_A(\varphi{\otimes}\varphi)}T_A(\overline{ f}{\otimes} g)\,(x,t,\nu,\xi)\nonumber\\
&=&\iint e^{-2\pi i(\widetilde{x}\nu+\widetilde{t}{\xi})}
T_A (\overline{ f}{\otimes} g )(\widetilde{x},\widetilde{t})
T_A (\varphi{\otimes}\varphi )(\widetilde{x}-x,\widetilde{t}-t)
\,d\widetilde{x}\,d\widetilde{t}\nonumber\\
&=&\int \Big(\int e^{-2\pi i\widetilde{t}\xi}\overline{ f}(\widetilde{x}-\widetilde{t})\varphi(\widetilde{x}-x-\widetilde{t}+t)
\,d\widetilde{t}\, \Big)e^{-2\pi i\widetilde{x}\nu}g(\widetilde{x})\varphi(\widetilde{x}-x)\,d\widetilde{x}\nonumber\\
&=&\iint\overline{ f}(s)g(\widetilde{x})e^{-2\pi i\nu \widetilde{x}-2\pi i\xi(\widetilde{x}-s)}
\varphi(s-(x-t))\varphi(\widetilde{x}-x)\,d\widetilde{x}\,ds\nonumber\\
&=&
\Big(\overline{\int e^{-2\pi i\xi s}f(s)\varphi(s-(x-t))\,ds}\Big)
\Big(\int e^{-2\pi i(\nu+\xi)\widetilde{x}}g(\widetilde{x})\varphi(\widetilde{x}-x)\,d\widetilde{x}\Big)
\nonumber\\
&=&\overline{ V_\varphi f (x-t,\xi)}\  V_{\varphi}g(x,\nu+\xi).\nonumber
\end{eqnarray}
\end{proof}

\begin{lemma}\label{VRfg}
Let $\varphi\in\S(\Rd)$ be a nonzero even real valued Schwartz function on $\Rd$.
Then for all $f$ and $g$ in $\S(\Rd)$
\begin{eqnarray}
&& V_{\overline{R(\varphi,\varphi)}}\overline{R(f,g)} \, (x,\xi,\nu,t)
=e^{-2\pi i\xi t} \ V_{T_{A}(\varphi{\otimes}\varphi)}T_{A}(\overline{ f}{\otimes} g)\,(x,-t,\nu,\xi)\,.\nonumber
\end{eqnarray}
\end{lemma}

\begin{proof}
For all $f$ and $g$ in $\S(\Rd)$
\begin{eqnarray}
&& \hspace{-1cm}V_{\overline{R(\varphi,\varphi)}}\overline{R(f,g)} (x,\xi,\nu,t)\nonumber\\
&=&\iint e^{-2\pi i(\nu\widetilde{x}+t\widetilde{\xi})}\overline{R(f,g)(\widetilde{x},\widetilde{\xi})}
{R(\varphi,\varphi)}(\widetilde{x}-x,\widetilde{\xi}-\xi)\,d\widetilde{x}\,d\widetilde{\xi}\nonumber\\
&=&\iint e^{-2\pi i(\nu\widetilde{x}+t\widetilde{\xi})}
\F_{\widetilde{t}\to\widetilde{\xi}}\big(\overline{ f}(\widetilde{x}-\cdot)\big)g(\widetilde{x})
\overline{\F_{\widetilde{t}\to\widetilde{\xi}-\xi}\big(\varphi(\widetilde{x}-x-\cdot)\big)}\varphi(\widetilde{x}-x)\,d\widetilde{x}\,d\widetilde{\xi}\nonumber\\
&=&\iint e^{-2\pi i(\nu\widetilde{x}+t\widetilde{\xi})}
\F_{\widetilde{t}\to\widetilde{\xi}}\big(\overline{ f}(\widetilde{x}-\cdot)\big)g(\widetilde{x})
{\F_{\widetilde{t}\to\xi-\widetilde{\xi}}\big(\varphi(\widetilde{x}-x-\cdot)\big)}\varphi(\widetilde{x}-x)\,d\widetilde{x}\,d\widetilde{\xi}.\nonumber\\
&&\label{lemA}
\end{eqnarray}
On the other hand,   Parseval's identity gives
\begin{eqnarray}
&&\hspace{-1cm} V_{{T_A(\varphi{\otimes}\varphi)}}{T_A(\overline{ f}{\otimes} g)} (x,t,\nu,\xi)\nonumber\\
&=&\iint e^{-2\pi i(\widetilde{x}\nu+\widetilde{t}{\xi})}
T_A (\overline{ f}{\otimes} g )(\widetilde{x},\widetilde{t})
T_A (\varphi{\otimes}\varphi )(\widetilde{x}-x,\widetilde{t}-t)
\,d\widetilde{x}\,d\widetilde{t}\nonumber\\
&=&\int \Big(\int e^{-2\pi i\widetilde{t}\xi}\overline{ f}(\widetilde{x}-\widetilde{t})\varphi(\widetilde{x}-x-\widetilde{t}+t)
\,d\widetilde{t}\Big)e^{-2\pi i\widetilde{x}\nu}g(\widetilde{x})\varphi(\widetilde{x}-x)\,d\widetilde{x}\nonumber\\
&=&\iint\F_{\widetilde{t}\to\widetilde{\xi}}\big(\overline{ f}(\widetilde{x}-\cdot)\big)
\F^{-1}_{\widetilde{t}\to\widetilde{\xi}}\big(e^{-2\pi i\widetilde{t}\xi}\varphi(\widetilde{x}-x+t-\cdot)\big)
e^{-2\pi i\widetilde{x}\nu} g(\widetilde{x})\varphi(\widetilde{x}-x)
\,d\widetilde{\xi}\,d\widetilde{x}\nonumber.
\end{eqnarray}
But,
$$\F^{-1}_{\widetilde{t}\to\widetilde{\xi}}\big(e^{-2\pi i\widetilde{t}\xi}\varphi(\widetilde{x}-x+t-\cdot)\big)=
e^{-2\pi it(\xi-\widetilde{\xi})}\F_{\gamma\to\xi-\widetilde{\xi}}\big(\varphi(\widetilde{x}-x-\cdot)\big),$$
therefore,
\begin{eqnarray}
 V_{{T_A(\varphi{\otimes}\varphi)}}{T_A(\overline{ f}{\otimes} g)} (x,t,\nu,\xi)&=&
e^{-2\pi i t\xi}\iint e^{2\pi i(t\widetilde{\xi}-v\widetilde{x})}
\F_{\widetilde{t}\to\widetilde{\xi}}\big(\overline{ f}(\widetilde{x}-\cdot)\big) \ \cdot \nonumber\\ && \qquad
\F_{\widetilde{t}\to\xi-\widetilde{\xi}}\big(\varphi(\widetilde{x}-x-\cdot)\big)
g(\widetilde{x})\varphi(\widetilde{x}-x)\,d\widetilde{x}\,d\widetilde{\xi}.\nonumber
\end{eqnarray}
Combining this identity with  (\ref{lemA})  completes the proof.
\end{proof}

\begin{proposition}\label{Vtilde}
Let $w_1,w_2,w$ be moderate
functions that satisfy
\begin{equation}\notag
w(x,t,\nu,\xi)\leq w_1(x-t,\xi)w_2(x,\nu+\xi).
\end{equation}
Let $\varphi$ be a nonzero real valued Schwartz function on  $\Rd$ and define
\begin{equation}\label{defVtilde}
 {\cal V}_{T_A(\varphi{\otimes}\varphi)}T_A(\overline{ f}{\otimes} g)\, (x,t,\xi,\nu)=
 V_{T_A(\varphi{\otimes}\varphi)}T_A(\overline{ f}{\otimes} g)\,  (x,t,\nu,\xi)
\end{equation}
for all $f,g\in\S(\Rd)$ and $x,t,\xi,\nu\in\Rd$.
If
$p_1,p_2,p_3,p_4,q_1,q_2,q_3,q_4\in[1, \infty]$ satisfy
\begin{equation}
\begin{array}{cc}
\frac{1}{p_1}+\frac{1}{p_2}=\frac{1}{p_3}+\frac{1}{p_4},& p_3\leq\min\{p_1,p_2,p_4\},\\
\frac{1}{q_1}+\frac{1}{q_2}=\frac{1}{q_3}+\frac{1}{q_4},
&
q_3\leq\min\{q_1,q_2,q_4\},
\end{array}\label{cond-lemma}
\end{equation}
then
$$\|{\cal V}_{T_A(\varphi{\otimes}\varphi)}T_A(\overline{ f}{\otimes} g)\|_{L_w^{p_3 p_4 q_3 q_4}}\leq
\|f\|_{M_{w_1}^{p_1 q_1}}\|g\|_{M_{w_2}^{p_2.q_2}}.
$$
\end{proposition}

\begin{proof}
By Lemma~\ref{lemmaT_A}, we have
$$ {\cal V}_{T_A(\varphi{\otimes}\varphi)}T_A(\overline{ f}{\otimes} g) \, (x,t,\xi,\nu)=\overline{ V_\varphi f (x-t,\xi)} \ V_{\varphi}g (x,\nu+\xi).$$
So, by (\ref{WeightCondition}), for $t,\xi,\nu\in\Rd$,
\begin{eqnarray}
&&\hspace{-1cm}\|w(\cdot,t,\xi,\nu){\cal V}_{T_A(\varphi{\otimes}\varphi)}T_A(\overline{ f}{\otimes} g)(\cdot,t,\xi,\nu)\|_{L^{p_3}}\nonumber\\
&\leq&\Big(\int  |w_1(x-t,\xi)\Big(V_\varphi f\Big)(x-t,\xi)|^{p_3}|w_2(x,\nu+\xi)\Big(V_{\varphi}g\Big)(x,\nu+\xi)|^{p_3}\,dx\Big)^{1/p_3}\nonumber\\
&=&\Big(  \left|w_2(\cdot,\nu+\xi)V_{\varphi}g(\cdot,\nu+\xi)\right|^{p_3}*\left|w_1(\cdot,\xi)
V_{\varphi}f(\cdot,\xi)\right|^{p_3} (t) \Big)^{1/p_3}\nonumber.
\end{eqnarray}
Then, \eqref{cond-lemma} implies
$$\frac{1}{r_1}+\frac{1}{s_1}=1+\frac{1}{a_1},$$
with $r_1=p_2/p_3\geq 1$, $s_1=p_1/p_3\geq 1$ and $a_1=p_4/p_3\geq 1$, hence, we can apply
  Young's inequality and obtain
\begin{eqnarray}
&&\hspace{-1cm}\|w(\cdot,\cdot,\xi,\nu){\cal V}_{T_A(\varphi{\otimes}\varphi)}T_A(\overline{ f}{\otimes} g)(\cdot,\cdot,\xi,\nu)\|_{L^{p_3,p_4}}\nonumber\\
&=&\left\|\left|w_2(\cdot,\nu+\xi)V_{\varphi}g(\cdot,\nu+\xi)\right|^{p_3}*\left|w_1(\cdot,\xi)V_{\varphi}f(\cdot,\xi)\right|^{p_3}\right\|_{L^{a_1}}^{1/p_3}\nonumber\\
&\leq& \left\|\left|w_2(\cdot,\nu+\xi)V_{\varphi}g(\cdot,\nu+\xi)\right|^{p_3}\right\|_{L^{r_1}}^{1/p_3}\ \left\|\left|w_1(\cdot,\xi)V_{\varphi}f(\cdot,\xi)\right|^{p_3}
\right\|_{L^{s_1}}^{1/p_3}.\label{estimate2}
\end{eqnarray}
To estimate (\ref{estimate2}) further, we note that integrating with respect to $\xi$ can
be again considered a convolution. In fact
 (\ref{cond-lemma}) leads to
$$\frac{1}{r_2}+\frac{1}{s_2}=1+\frac{1}{a_2},$$
where $r_2=q_2/q_3$, $s_2=q_1/q_3$ and $a_2=q_4/q_3$.
   Young's inequality then implies
\begin{eqnarray}
&& \hspace{-1cm} \|w{\cal V}_{T_A(\varphi{\otimes}\varphi)}T_A(\overline{ f}{\otimes} g)\|_{L^{p_3 p_4 q_3 q_4}}\nonumber\\
&\leq&\Big(\int \Big(\int \left|w_2(x,y)V_{\varphi}g(x,y)\right|^{p_3r_1}\,dx\Big)^{(r_2q_3)/(p_3r_1)}\,dy\Big)^{1/(r_2q_3)}\cdot \nonumber\\
&&\quad \Big(\int \Big(\int \left|w_1(x,y)V_{\varphi}f(x,y)\right|^{p_3s_1}\,dx\Big)^{(s_2q_3)/(p_3s_1)}\,dy\Big)^{(1/s_2q_3)}\nonumber\\
&=&\|f\|_{M_{w_1}^{p_1 q_1}}\|g\|_{M_{w_2}^{p_2 q_2}},\nonumber
\end{eqnarray}
which completes the proof.
\end{proof}

Now, we are ready to give sufficient conditions on the boundedness of  pseudo-differential operators  with
symbols in $\TM^{p_3,p_4,q_3.q_4}(\Rtd)$.

\begin{lemma}\label{necess}
Let $w_1,w_2,w$ be moderate weight functions that satisfy (\ref{WeightCondition}). Let
$p_1,p_2,p_3,p_4,q_1,q_2,q_3,q_4\in[1, \infty]$ be such that
\begin{equation}\label{condition1}
\begin{array}{ll}
\frac{1}{p_3}\in\left[\frac{1}{p'_1}+\frac{1}{p_2}-\frac{1}{p_4},\min\{\frac{1}{p'_1},\frac{1}{p_2},\frac{1}{p_4}\}\right],\\
\frac{1}{q_3}\in\left[\frac{1}{q'_1}+\frac{1}{q_2}-\frac{1}{q_4}, \min\{\frac{1}{q'_1},\frac{1}{q_2},\frac{1}{q_4}\}\right].
\end{array}
\end{equation}
Then  there exists a  constant $C>0$ such that
\begin{equation}\label{BoundedMs}
\|T_{\sigma}\|_{\L(M_{w_1}^{p_1 q_1}, M_{w_2}^{p_2 q_2})}\leq C\, \|\sigma\|_{\TM_{w}^{p_3 p_4 q_3 q_4}},\quad \sigma\in \TM_{w}^{p_3 p_4 q_3 q_4}(\Rtd).
\end{equation}
\end{lemma}

\begin{proof}
Let us first assume  $p_1,p_2,p_3,p_4,q_1,q_2,q_3,q_4\in[1, \infty]$ satisfy (\ref{condition1}) and in addition
\begin{equation}\label{case1}
\frac{1}{p'_1}+\frac{1}{p_2}=\frac{1}{p_3}+\frac{1}{p_4}\quad \text{and} \quad
\frac{1}{q'_1}+\frac{1}{q_2}=\frac{1}{q_3}+\frac{1}{q_4}.
\end{equation}
Let $f,g\in\S(\Rd)$. Since the dual of $\TM_{w}^{p_3 p_4 q_3 q_4}(\Rtd)$
is $\TM_{w}^{p'_3,p'_4,q'_3,q'_4}(\Rtd)$, it follows  that
\begin{eqnarray}
|(T_\sigma f,g)|&=&|(\sigma,\overline{R(f,g)})|\nonumber\\
&\leq&\|\sigma\|_{\TM_{w}^{p_3 p_4 q_3 q_4}}\|\overline{R(f,g)}\|_{\TM_{w}^{p'_3,p'_4,q'_3,q'_4}}.\nonumber
\end{eqnarray}
To obtain (\ref{BoundedMs}), it  is enough to  show that there exists $C>0$ such that
$$\|\overline{R(f,g)}\|_{\TM_{w}^{p'_3,p'_4,q'_3,q'_4}}\leq C\, \|f\|_{M_{w_1}^{p_1 q_1}}\|g\|_{M_{w_2}^{p'_2,q'_2}}.$$
Let $\varphi$ be a nonzero  real valued even function in $\S(\Rd)$. Then by Lemma~\ref{VRfg},
\begin{eqnarray}
\Big| {\cal V}_{\overline{R(\varphi,\varphi)}}\overline{R(f,g)}\,(x,t,\xi,\nu)\Big|&=&
\Big| V_{\overline{R(\varphi,\varphi)}}\overline{R(f,g)}\, (x,\xi,\nu,-t)\Big|\nonumber\\
&=&\Big| V_{T_{A}(\varphi{\otimes}\varphi)}T_{A}(\overline{ f}{\otimes} g)\,(x,t,\nu,\xi)\Big|\nonumber\\
&=&\Big| {\cal V}_{T_{A}(\varphi{\otimes}\varphi)}T_{A}(\overline{ f}{\otimes} g)\,(x,t,\xi,\nu)\Big|.\nonumber
\end{eqnarray}
where ${\cal V}_{T_{A}(\varphi{\otimes}\varphi)}$ is defined in (\ref{defVtilde}). Therefore, by Proposition~\ref{Vtilde}, we have
\begin{eqnarray}
 \|R(f,g)\|_{\TM_{w}^{p'_3,p'_4,q'_3,q'_4}}&=&\|{\cal V}_{T_{A}(\varphi{\otimes}\varphi)}T_{A}(\overline{ f}{\otimes} g)\|_{L_w^{p'_3,p'_4,q'_3,q'_4}}\nonumber\\
 &\leq&\|f\|_{M_{w_1}^{p_1 q_1}}\|g\|_{M_{w_2}^{p'_2,q'_2}}.\nonumber
 \end{eqnarray}
 To obtain (\ref{BoundedMs}) in the general case, that is $p_1,p_2,p_3,p_4,q_1,q_2,q_3,q_4\in[1, \infty]$ satisfy (\ref{condition1}) but not
 necessarily (\ref{case1}), set
 $$\frac{1}{\widetilde{p}_2}=\frac{1}{p_3}+\frac{1}{p_4}-\frac{1}{p'_1}\quad \text{and} \quad \frac{1}{\widetilde{q}_2}=\frac{1}{q_3}+\frac{1}{q_4}-\frac{1}{q'_1}.$$
 Then it is easy to see that $\widetilde{p}_2\leq p_2$, $\widetilde{q}_2\leq q_2$ and $p_1,\widetilde{p}_2,p_3,p_4,q_1,
 \widetilde{q}_2,q_3,q_4\in[1, \infty]$ satisfy (\ref{condition1}). Hence
 $$\|T_{\sigma}f\|_{M_{w_2}^{p_2 q_2}}\leq C\, \|T_\sigma f\|_{M_{w_2}^{\widetilde{p}_2,\widetilde{q}_2}}\leq \|f\|_{M_{w_1}^{p_1 q_1}}
 \|\sigma\|_{\TM_{w}^{p_3 p_4 q_3 q_4}},$$
 for some $C>0$.
\end{proof}

\noindent
{\bf{Proof of Theorem~\ref{MainTheoremWeight}:}}
Let $f\in\S(\Rd)$. Set
$$\frac{1}{\widetilde{p}_3}=\frac{1}{p'_1}+\frac{1}{p_2}-\frac{1}{p_4}\quad \text{and} \quad \frac{1}{\widetilde{q}_3}=\frac{1}{q'_1}+\frac{1}{q_2}-\frac{1}{q_4}.$$
Then  it is easy to see that
$$\widetilde{p}_3\geq p_3,\quad  \widetilde{q}_3\geq q_3.$$
 Furthermore,
$\{p_1,p_2,\widetilde{p}_3,p_4,q_1,q_2,\widetilde{q}_3,q_4\}$ satisfies (\ref{condition1}),
therefore there exist  $C_1,C_2>0$ such that
\begin{eqnarray}
\|T_{\sigma}f\|_{M_{w_2}^{p_2 q_2}}&\leq& C_1 \|f\|_{M_{w_1}^{p_1 q_1}}
 \|\sigma\|_{\TM_{w}^{\widetilde{p}_3,p_4,\widetilde{q}_3,q_4}}\nonumber\\
 &\leq& C_2  \|f\|_{M_{w_1}^{p_1 q_1}}
 \|\sigma\|_{\TM_{w}^{p_3 p_4 q_3 q_4}}\nonumber.
 \end{eqnarray}
%

\subsection{Proof of Corollary~\ref{Corollary}}
Let  $1\leq p,q\leq 2$. By Theorem~\ref{MainTheorem}, $T_\sigma: M^{p,p'}\to M^q$ is bounded. Using the bounded embeddings  $M^p\subset L^p\subset M^{p,p'}$ for all  $1\leq p\leq 2$ (for more details see \cite{Fei1}), it follows that
 $T_\sigma:L^p(\Rd)\to L^q(\Rd)$ is bounded. Similarly, using $M^{p,p'}\subset L^p\subset M^p$
 for all  $q\geq 2$, we can prove
$T_\sigma:L^p(\Rd)\to L^q(\Rd)$ is bounded for $p,p_3,p_4,q,q_3,q_4$ satisfying (b) or (c) or (d) in Corollary~\ref{Corollary}.\hfill $\square$


\subsection{Proof of   (\ref{MainTheorem1}) implies (\ref{MainTheorem2}), (\ref{MainTheorem3}) in Theorem~\ref{MainTheorem}}
To show necessity of \eqref{MainTheorem2} and \eqref{MainTheorem3} in  Theorem~\ref{MainTheorem}, we shall use two mixed $L^p$ norms on phase space, namely,
$$\|F\|_{L^{pq}}=\Big(\int\Big(\int|F(x,\xi)|^{p}\,dx\Big)^{q/p}\,d\xi\Big)^{1/q},$$
and
$$\|F\|_{\widetilde{L}^{pq}}=\Big(\int\Big(\int|F(x,\xi)|^{q}\,d\xi\Big)^{p/q}\,dx\Big)^{1/p},$$
for $p,q\in[1,\infty)$. For $p=\infty$ and/or $q=\infty$ we make  the usual adjustment.

Similarly,
we can define $\widetilde{M}^{pq}(\Rd)$ to be the space of all functions $f\in\S'(\Rd)$ for which
$$\|f\|_{\widetilde{M}^{pq}}=\|V_{\varphi}f\|_{\widetilde{L}^{pq}}<\infty,$$
where $\varphi\in\S(\Rd)\setminus\{0\}$.
Note that it can be easily checked that
$$\|f\|_{\widetilde{M}^{pq}}=\|\widehat f\|_{{M}^{q p}}.$$
Below, we use an idea from the proof of Proposition~\ref{compactMpq} given in  \cite{Oko}
 to prove the following lemma.
\begin{lemma}\label{compact support}
Let $K\subset\Rtd$ be compact. Then
 $$\|\sigma\|_{\TM^{p_3 p_4 q_3 q_4}}\asymp\|\sigma\|_{\F\widetilde{L}^{q_4 p_4}},\quad \sigma\in\S'(\Rtd),\quad \supp \sigma\subset K.$$
\end{lemma}

\begin{proof}
Choose $r>0$ with  $\supp \sigma\subseteq B^{2d}_r(0)$, where
$$B^{2d}_r(0)=\{x\in\Rtd:\|x\|\leq r\}$$
 is the Euclidean unit ball in $\Rtd$ with center 0, radius $r$ and Lebesgue measure $|B^{2d}_r(0)|$. Let $\psi\in\S(\Rtd)$ with
$supp\ \psi\subset B_r^{2d}(0)$.
 Then it is easy to see that
$$\big|\TV_\psi \sigma\big|(x,t,\xi,\nu)=
\big|\sigma*M_{\nu,-t}\overline{\widetilde{\psi}}\big|(x,\xi),
$$
where $$\widetilde{\psi}(x,\xi)=\psi(-x,-\xi) .$$
Therefore, for fixed $t,\nu$ we have
\begin{eqnarray}
\supp \big(\left|\TV_\psi \sigma\right|(\cdot,t,\cdot,\nu)\big)
&\subseteq&\supp (\sigma)+supp\ (M_{\nu,-t}\overline{\widetilde{\psi}})\nonumber\\
&\subseteq& B^{2d}_r(0)+B^{2d}_r(0)\subseteq B^{2d}_{2r}(0).\label{support1}
\end{eqnarray}
Let $\xi\in B^d_{2r}(0)$. Then by (\ref{support1}),
\begin{eqnarray}
&&\hspace{-1cm} \|\TV_\psi \sigma(\cdot,t,\xi,\nu)\|^{p_3}_{L^{p_3}(\Rd)}=
\int_{B^d_{2r}(0)}\big|\TV_\psi \sigma\big|^{p_3}(x,t,\xi,\nu)\,dx\nonumber\\
&\leq&|B_{2r}^d(0)|\ \big\|\TV_\psi \sigma(\cdot,t,\xi,\nu)\big\|_{L^\infty} =|B_{2r}^d(0)|\ \big\|\sigma*M_{\nu,-t}\overline{\widetilde{\psi}}(\cdot,\xi)\big\|_{L^\infty}\nonumber\\
&\leq&|B_{2r}^d(0)|\ \big\|\sigma*M_{\nu,-t}\overline{\widetilde{\psi}}\big\|_{L^\infty}\leq  |B_{2r}^d(0)|
\ \big\|\F^{-1}\big(\widehat{\sigma}\, T_{\nu,-t}\overline{\widehat{\psi}}\big)\big\|_{L^\infty}\nonumber\\
&\leq&|B_{2r}^d(0)|\ \big\|\widehat{\sigma}T_{\nu,-t}\overline{\widehat{\psi}}\big\|_{L^1}=|B_{2r}^d(0)|\,\big(|\widehat{\sigma}|*|{\widehat{\psi}}|\big)(-\nu,t)\label{support2}
\end{eqnarray}
On the other hand, if $\xi\in\Rd\setminus B^d_{2r}(0)$, then by (\ref{support1}),
\begin{equation}\label{support3}
\|\TV_\psi \sigma(\cdot,t,\xi,\nu)\|_{L^{p_3}}=0.
\end{equation}
Therefore,  (\ref{support2}) and (\ref{support3}) imply
\begin{eqnarray}
&&\hspace{-1cm}\|\sigma\|_{\TM^{p_3 p_4 q_3 q_4}}=\left\|\TV_\psi \sigma\right\|_{L^{p_3 p_4 q_3 q_4}}\nonumber\\
&\leq&|B_{2r}^d(0)|^{1/p_3}
\int
\Big(\int_{B^d_{2r}(0)}
\Big(\int
\Big((|\widehat{\sigma}|*|{\widehat{\psi}}|(-\nu,t))^{p_4}\,dt
\Big)^{q_3/p_4}
\,d\xi\Big)^{q_4/q_3}
\,d\nu
\Big)^{1/q_4}\nonumber\\
&\leq&|B_{2r}^d(0)|^{(1/p_3)+(1/q_3)}
\Big(
\int
\Big(
\int
(|\widehat{\sigma}|*|{\widehat{\psi}}|(-\nu,t))^{p_4}\,dt
\Big)^{q_4/p_4}
\,d\nu
\Big)^{1/q_4}\nonumber\\
&\leq&|B_{2r}^d(0)|^{(1/p_3)+(1/q_3)}\left\||\widehat{\sigma}|*|{\widehat{\psi}}|\right\|_{\widetilde{L}^{q_4,p_4}}\nonumber\\
&\leq&|B_{2r}^d(0)|^{(1/p_3)+(1/q_3)}\|\widehat\sigma\|_{\widetilde{L}^{q_4,p_4}}\|\widehat\psi\|_{\widetilde{L}^{1,1}}
\leq C\|\widehat\sigma\|_{\widetilde{L}^{q_4,p_4}}\nonumber.
\end{eqnarray}
Now,  let $\psi\in C^{\infty}(\Rtd)$ be compactly supported  with
$\psi\equiv 1$ on $B_{2r}^{2d}(0)$.
Let $\chi_{B_r^{2d}(0)}$ be  the characteristic function on $B_r^{2d}(0)$.
 Then using  $supp\ \sigma\subseteq
B_r^{2d}(0)$, it follows that  for all $x,t,\xi,\nu\in\Rd$,
\begin{eqnarray}
&&\hspace{-1cm}\chi_{B_{2r}^{2d}(0)}(x,\xi)\ \TV_{\psi}\sigma (x,t,\xi,\nu)\nonumber\\
&=&\chi_{B_{2r}^{2d}(0)}(x,\xi)\int_{B_r^{2d}(0)}\sigma(\widetilde{x},\widetilde{\xi})\, e^{-2\pi i(\widetilde{x}\nu-\widetilde{\xi}t)}\,
\psi(\widetilde{x}-x,\widetilde{\xi}-\xi)\,d\widetilde{x}\,d\widetilde{\xi}\nonumber\\
&=&\chi_{B_{2r}^{2d}(0)}(x,\xi)\int_{B_r^{2d}(0)}\sigma(\widetilde{x},\widetilde{\xi})\,e^{-2\pi i(\widetilde{x}\nu-\widetilde{\xi}t)}\,
\,d\widetilde{x}\,d\widetilde{\xi}\nonumber\\
&=&\chi_{B_{2r}^{2d}(0)}(x,\xi)\F\sigma(\nu,-t).\nonumber
\end{eqnarray}
Hence,
\begin{eqnarray*}
&& \hspace{-1cm}\|\sigma\|_{\TM^{p_3 p_4 q_3 q_4}}=\big\|\TV_{\psi}\sigma\big\|_{L^{p_3 p_4 q_3 q_4}}\ \geq \
    \big\|\chi_{B_{2r}^{2d}(0)}\, \TV_{\psi}\sigma\big\|_{L^{p_3 p_4 q_3 q_4}}
\\
&=&\Big(\int\Big(\int\Big(\int\Big(\int \big| \chi_{B_{2r}^{2d}(0)}(x,\xi)\, \F\sigma(\nu,-t)\big|^{p_3} dx \Big)^{p_4/p_3}dt\Big)^{q_3/p_4}d\xi\Big)^{q_4/q_3}d\nu\Big)^{1/q_4}
\\
&=& \big\| \chi_{B_{2r}^{2d}(0)}\|_{L^{p_3q_3}}\ \|\sigma\|_{\F\widetilde{L}^{q_4,p_4}}
\end{eqnarray*}
which completes the proof.
\end{proof}

\begin{lemma}\label{lambdaa}
Let $\lambda>0$ and
$\varphi_{\lambda}(x)=e^{-\pi \lambda|x|^2}$.
 Then for $\lambda\geq 1$,
$$\|\varphi_{\lambda}\|_{M^{pq}}\asymp \|\varphi_{\lambda}\|_{\widetilde{M}^{pq}}\asymp \lambda^{-d/q'}, $$
and
$$\|\varphi_{\lambda^{-1}}\|_{M^{pq}}\asymp \|\varphi_{\lambda^{-1}}\|_{\widetilde{M}^{pq}}\asymp \lambda^{d/p}.$$
\end{lemma}
 The proof of Lemma~\ref{lambdaa} is an immediate corollary of
 Lemma 3.2 in \cite{CorderoNicolaMeta} and  is omitted here.
\begin{lemma}\label{hlambda}
Let $K\subset \Rd$ be compact. For $h\in C^{\infty}(\Rd)$ and  $\lambda\geq 1$ set
$h_\lambda(x)=h(x)e^{-\pi i\lambda|x|^2}$.
Then for all $p,q\in [1,\infty]$,
$$\|h_\lambda\|_{M^{pq}}\asymp\|\widehat h_\lambda\|_{L^q}\asymp\lambda^{d/q-d/2},\quad h\in C^{\infty}(\Rd),\quad\supp h\subset K.$$
\end{lemma}
Lemma~\ref{hlambda} is well known and its proof can be found in, for example, \cite{CorderoNicolaPseudo}.
\begin{lemma}\label{eta}
Let $h_1,h_2\in\S(\Rd)$ and
$$\eta(t,\nu)=e^{-2\pi i t\nu}h_2(t)\widehat h_1(\nu),\quad t,\nu\in\Rd.$$
If
$\sigma=\TF \eta$. Then  we have
\begin{eqnarray}
  \label{anotherone}\sigma(x,\xi)= (M_\xi h_2*h_1 )(x)
\end{eqnarray}
and
\begin{eqnarray}
  \label{anothertwo} T_\sigma f= (h_1f )*h_2,\quad f\in\S(\Rd).
\end{eqnarray}
Moreover,
$$\|\sigma\|_{\TM^{p_3 p_4 q_3 q_4}}=\|h_1\|_{M^{p_3,q_4}}\|h_2\|_{M^{p_4,q_3}}.$$

\end{lemma}

\begin{proof}
Clearly, \eqref{anotherone} and \eqref{anothertwo} hold.
Now, let $\varphi$ be any nonzero real valued Schwartz function on $\Rd$. Let
$$\psi(t,\nu)=\varphi(t)\varphi(\nu)e^{-2\pi it\nu} .$$
and define
$$\widetilde{\psi}(x,\xi)= \TF \psi (-x,-\xi).$$
Then
\begin{eqnarray}
\left|\TV_{\widetilde{\psi}}\sigma(x,t,\xi,\nu)\right|&=&\left|\big(\sigma,M_{\nu.-t}T_{x,\xi}\widetilde{\psi}\big)\right|\nonumber\\
&=&\left|\big(\TF \eta,\TF (T_{-t,\nu}M_{-\xi,x}\psi)\big)\right|.\nonumber
\end{eqnarray}
Now since $\TF $ is a unitary operator, it follows that
\begin{eqnarray}
&&\hspace{-1cm}\left|\Big(\TV_{\widetilde{\psi}}\sigma\Big)(x,t,\xi,\nu)\right|
=\left|\Big(\eta,T_{-t,\nu}M_{-\xi,x}\psi\Big)\right|\nonumber\\
&=&\left|\iint\eta(\widetilde{t},\widetilde{\nu})\, e^{2\pi i\xi(\widetilde{t}+t)}\, e^{-2\pi ix(v-\widetilde{\nu})}
\overline{\psi}(t+\widetilde{t},\widetilde{\nu}-\nu)\,d\widetilde{t}\,d\widetilde{\nu}
\right|\nonumber\\
&=&\left|\iint\widehat h_1(\widetilde{\nu})h_2(\widetilde{t})\varphi(\widetilde{\nu}-\nu)\varphi(\widetilde{t}+t)\,
e^{-2\pi i\widetilde{\nu}(x-t)}\, e^{2\pi i\widetilde{t}(\nu-\xi)}\,d\widetilde{t}\,d\widetilde{\nu}\right|\nonumber\\
&=&|(V_{\varphi}\widehat h_1)(\nu,x-t)|\  |(V_{\varphi}h_2)(-t,\nu-\xi)|.\nonumber
\end{eqnarray}
Hence,
$$\|\sigma\|_{\TM^{p_3 p_4 q_3 q_4}}=\|h_1\|_{M^{p_3,q_4}}\|h_2\|_{M^{p_4,q_3}}.$$
\end{proof}

Similarly, we can prove the following.
\begin{lemma}\label{h1h2}
Let Let $h_1,h_2\in\S(\Rd)$ and  $\sigma=h_1\otimes\widehat{h}_2.$ Then $$T_{\sigma} f=h_1\cdot(h_2*f),\quad f\in \S(\Rd)$$ and
$$\|h_1\otimes\widehat{h}_2\|_{\TM^{p_3 p_4 q_3 q_4}}=\|h_1\|_{M^{p_3,q_4}}\|\widehat{h}_2\|_{\widetilde{M}^{q_3,p_4}}.$$
\end{lemma}

\noindent
{\bf Proof of   (\ref{MainTheorem1}) implies (\ref{MainTheorem2}) and (\ref{MainTheorem3}) in Theorem~\ref{MainTheorem}:}
Let $h\in C^{\infty}(\Rd)$ be chosen with compact support and  $h(0)=1$ and $h(x)\geq 0$ for all $x\in\Rd$. Then for any
$\lambda\geq 1$, we define $h_\lambda$ and $\sigma_\lambda$  respectively by
$$h_\lambda(x)=h(x)e^{-\pi i \lambda|x|^2}.$$
and
$$\sigma_{\lambda}(x,\xi)=h{\otimes} h_\lambda (x,\xi)=h(x)h_\lambda(\xi).$$
Let $f_\lambda=\F^{-1}\overline{ h_\lambda}$. Then $f_\lambda\in\S(\Rd)$ and
$$ T_{\sigma_\lambda}f_\lambda (x)=\int  e^{2\pi ix\xi}h(x)|h(\xi)|^2\,d\xi .$$
So, $T_{\sigma_\lambda}f_\lambda$ is independent of $\lambda$. Since $\sigma_\lambda$ has compact
support, by Lemma~\ref{compact support} and Lemma~\ref{hlambda}
\begin{eqnarray}
\|\sigma_\lambda\|_{\TM^{p_3 p_4 q_3 q_4}}&\asymp&\|\F\sigma_\lambda\|_{\widetilde{L}^{q_4,p_4}}\nonumber\\
&=&\|\widehat h\|_{L^{q_4}(\Rd)}\|\widehat h_\lambda\|_{L^{p_4}(\Rd)}\nonumber\\
&\asymp& \lambda^{(d/p_4)-(d/2)}.\label{bounded3}
\end{eqnarray}
Moreover,  by  Lemma~\ref{compactMpq} and Lemma~\ref{hlambda},  since  $\F f_\lambda$ has  compact support,
\begin{equation}\label{bounded4}
\|f_\lambda\|_{M^{p_1 q_1}(\Rd)}=\|f_\lambda\|_{L^{p_1}(\Rd)}\asymp \lambda^{(d/p_1)-(d/2)}.
\end{equation}
Hence by (\ref{MainTheorem1}), (\ref{bounded3}) and (\ref{bounded4}), there exists $C>0$ such that
for all $\lambda\geq 1$
$$
\|T_{\sigma}f_\lambda\|_{M^{p_2 q_2}(\Rd)}\leq C\,  \lambda^{(d/p_4)+(d/p_1)-d}.$$
But $\|T_{\sigma}f_\lambda\|_{M^{p_2 q_2}(\Rd)}$ is nonzero and independent of $\lambda$, therefore
$\frac{d}{p_4}+\frac{d}{p_1}-d\geq 0$, and $p_4\leq p'_1$.\\
To prove $q_4\leq q'_1$,  we let $h_1=\overline{ f}=h_\lambda$ and $h_2\in\S(\Rd)$
be such that $\widehat h_2$ is compactly supported, independent of $\lambda$ and
$$\|(h_1f)*h_2\|_{L^{p_2}(\Rd)}\not =0.$$
Let $\sigma=\TF \eta$ where
 $$\eta(t,\nu)=\widehat h_1(\nu)h_2(t)e^{-2\pi i t\nu}.$$
Then by Lemma~\ref{eta} and (\ref{MainTheorem1})
$$\|(h_1f)*h_2\|_{L^{p_2}(\Rd)}\leq C\,  \|\widehat h_1\|_{L^{q_4}(\Rd)}\|h_2\|_{L^{p_4}(\Rd)}\|\widehat f\|_{L^{q_1}(\Rd)},$$
for some constant $C>0$. So, by Lemma~\ref{hlambda} for all $\lambda\geq 1$
$$\|(h_1f)*h_2\|_{L^{p_2}(\Rd)}\leq C\, \lambda^{(d/q_4)-(d/2)}\lambda^{(d/q_1)-(d/2)},$$
but $\|(h_1f)*h_2\|_{L^{p_2}(\Rd)}$ is nonzero and independent of $\lambda$,
therefore $(d/q_4)+(d/q_1)-d\geq 0$ and, hence, $q_4\leq q'_1$.

Now, let $h_1=f=\varphi_\lambda$ and $h_2=\varphi_{\lambda^{-1}}$, where $\varphi_\lambda$ and
$\varphi_{\lambda^{-1}}$ are defined in Lemma~\ref{lambdaa}. If we let $\sigma=h_1\otimes h_2$.
 Then by  Lemma~\ref{lambdaa} and Lemma~\ref{h1h2}, for $\lambda\geq 1$ we have
 $$\|\sigma\|_{\TM^{p_3 p_4 q_3 q_4}(\Rtd)}\asymp\lambda^{d/q_3-d/q'_4}$$
 and $\|f\|_{M^{p_1 q_1}(\Rd)}\asymp \lambda^{-d/q'_1}$. On the other hand
 $T_{\sigma}f$ is also a Gaussian function and it can be easily checked that
 $$\|T_\sigma f\|_{M^{p_2 q_2}(\Rd)}\asymp \lambda^{-d/q'_2}.$$ Therefore by (\ref{MainTheorem1})
 $$\lambda^{d/q_3-d/q'_4-d/q'_1+d/q'_2}\geq 1$$
 for all $\lambda\geq 1$. Hence, we get
$$\frac{1}{q'_1}+\frac{1}{q_2}\leq\frac{1}{q_3}+\frac{1}{q_4}.$$
Similarly, by letting $h_1=f=\varphi_{\lambda^{-1}}$ and $h_2=\varphi_{\lambda}$, we get
$$\frac{1}{p'_1}+\frac{1}{p_2}\leq\frac{1}{p_3}+\frac{1}{p_4}.$$
Again assume    $\sigma$  has the form given in Lemma~\ref{eta}.
Let $h(x)=f(x)=e^{-\pi|x|^2/2}$ and $h_2=\varphi_{\lambda^{-1}}$. Then
$T_{\sigma}$ is also a Gaussian function, moreover by Lemma~\ref{lambdaa} and
(\ref{MainTheorem1})  for all $\lambda\geq 1$
$$\lambda^{d/p_4-d/p_2}\geq C,$$
for some $C>0$. Hence $p_4\leq p_2$.
\\
To prove $q_4\leq q_2$, we let
$$\sigma(x,\xi)=e^{2\pi i x\xi}h_1(x)h_2(\xi),$$
where $h_1$ and $h_2$ are compactly supported  Schwartz functions on $\Rd$.
Then $\sigma$ is compactly supported and therefore
 by Lemma~\ref{compact support},
 $$\|\sigma\|_{\TM^{p_3 p_4 q_3 q_4}(\Rtd)}=\|\F\sigma\|_{{L}^{p_4,q_4}(\Rtd)}.$$
On the other hand, by an easy calculation, we have
 $$\big|\F\sigma\big|(\nu,t)=\big|V_{h_1}\widehat h_2\big|(t,\nu)=
 \big|V_{\widehat{\overline{h_2}}}\overline{h_1}\big|(t,\nu).$$
 Therefore,
\begin{equation}\label{q_4}
\|\sigma\|_{\TM^{p_3 p_4 q_3 q_4}(\Rtd)}\leq C_{h_2}\, \|\widehat h_1\|_{L^{q_4}(\Rd)},
\end{equation}
and
$$\|\sigma\|_{\TM^{p_3 p_4 q_3 q_4}(\Rtd)}\leq C_{h_1}\, \|\widehat h_2\|_{L^{p_4}(\Rd)},$$
where $C_{h_1}$ and $C_{h_2}$ are positive constants depending on $h_1$ and $h_2$ respectively.
 Let $h_1=h_{\lambda}$ and $h_2$ be any compactly supported function
 and $f$ be a Schwartz function on $\Rd$
and  both $h_2$ and $\widehat f$ be independent of $\lambda$ such that $(h_2,\overline{\widehat f})\not =0$.
Then
\begin{eqnarray}
\|T_{\sigma}f\|_{M^{p_2 q_2}(\Rd)}&=&\|h_1\|_{M^{p_2 q_2}(\Rd)}|(h_2,\overline{\widehat f})|\nonumber\\
&=&|(h_2,\overline{\widehat f})|\|\widehat h_1\|_{L^{q_2}(\Rd)}\asymp \lambda^{(d/q_2)-(d/2)},\label{q_2}
\end{eqnarray}
and by (\ref{q_4})
$$\|\sigma\|_{\TM^{p_3 p_4 q_3 q_4}(\Rtd)}\leq C_{h_2}\, \lambda^{(d/q_4)-(d/2)}.$$
Hence, (\ref{q_2}) and (\ref{MainTheorem1}) imply
$$\lambda^{(d/q_4)-(d/q_2)}\geq C,$$
where $C>0$ is   independent of $\lambda\geq 1$. Hence
$(d/q_4)-(d/q_2)\geq 0$ which implies that $q_4\leq q_2$.\hfill $\square$



\vspace*{0.5cm}


\end{document}